\documentclass[11pt]{amsart}
\usepackage{amssymb}
\usepackage{amsmath}
\usepackage{tikz}
\usepackage{hyperref}
\usepackage[capitalize]{cleveref}
\usepackage{geometry}
\usepackage[spacing=true,kerning=true,babel=true,tracking=true]{microtype}
\usepackage[foot]{amsaddr}

\hypersetup{
	pdfstartview={XYZ null null 1.00}, 
	pdfpagemode=UseNone, 
	colorlinks,
	breaklinks, 
	linkcolor=blue,
	urlcolor=blue, 
	anchorcolor=blue,
	citecolor=blue
}

\newtheorem{theorem}{Theorem}
\newtheorem{lemma}[theorem]{Lemma}

\newtheorem{cor}[theorem]{Corollary}

\theoremstyle{definition}

\newtheorem{question}[theorem]{Question}

\crefname{question}{Question}{Questions}
\crefname{cor}{Corollary}{Corollaries}

\newcommand{\N}{\mathbb{N}}

\title{Randomized Borel $(2d+1)$-coloring of digraphs}

\author{Anton~Bernshteyn}
\address{\normalfont (AB) Department of Mathematics, University of California, Los Angeles, CA, USA}
\email{bernshteyn@math.ucla.edu}

\author{Edward Hou}
\address{\normalfont (EH) Department of Mathematics, California Institute of Technology, Pasadena, CA, USA}
\email{ehou@caltech.edu}

\author{Forte Shinko}
\address{\normalfont (FS) Department of Mathematics, University of California, San Diego, CA, USA}
\email{fshinko@ucsd.edu}

\author{Felix Weilacher}
\address{\normalfont (FW) Department of Mathematics, University of California, Berkeley, CA, USA}
\email{weilacher@berkeley.edu}

\thanks{AB's research is partially supported by the NSF CAREER grant DMS-2528522 and the Sloan Research Fellowship.
FW's research is supported by NSF grant DMS-2402064. }

\begin{document}

\begin{abstract}
Let $G$ be a Borel digraph with maximum out-degree $d \in \N$. 
We show that $G$ admits a random Borel $(2d+1)$-coloring for which every edge is almost surely not monochromatic. This gives a simpler proof of a recent result of Pelayo-G\'omez: such a graph $G$ admits a measurable proper $(2d+1)$-coloring with respect to any Borel probability measure on $V(G)$. 
Our proof is an adaptation of Pelayo-G\'omez's proof to the randomized Borel setting. 
\end{abstract}

\maketitle

\vspace{-12 pt}

\noindent Let $d \in \N$, and $G$ an (irreflexive) digraph where every vertex has out-degree at most $d$. A standard observation is that the underlying simple graph of $G$ is then $2d$-degenerate, and hence it admits a proper $(2d+1)$-coloring. This is optimal since the complete graph $K_{2d+1}$ has a balanced orientation.

In their seminal paper, Kechris, Solecki, and Todor\v{c}evi\'c considered the Borel combinatorics of Borel digraphs maximum out-degree $d \in \N$ (equivalently, graphs generated by $d$ Borel functions).

\begin{theorem}[Kechris--Solecki--Todor\v{c}evi\'c \cite{kst}]\label{thm:kst_main}
    Let $G$ be a Borel digraph with maximum out-degree 1.
    Then $\chi_B(G) \in \{0,1,2,3,\aleph_0\}$.  
\end{theorem}

\begin{cor}[Kechris--Solecki--Todor\v{c}evi\'c \cite{kst}]\label{cor:kst}
    Let $G$ be a Borel digraph with maximum out-degree $d \in \N$.
    Then $\chi_B(G) \in \{0,\ldots,3^d, \aleph_0\}$. 
\end{cor}

Miller showed that, modulo a meager or null set, these graphs always have finite Borel colorings:

\begin{theorem}[Miller \cite{miller_measurable}]\label{thm:miller}
    Let $G$ be a Borel digraph with maximum out-degree $d \in \N$. 
    \begin{itemize}
        \item For any Borel probability measure $\mu$ on $V(G)$, $\chi_\mu(G) \leq 3^d$. 
        \item For any compatible Polish topology on $V(G)$,
    $\chi_{BM}(G) \leq 3^d$.
    \end{itemize} 
\end{theorem}

Kechris, Solecki, and Todor\v{c}evi\'c asked whether the $3^d$ in their result could be replaced by the classical bound of $2d+1$. 
Kechris and Marks asked the same question 
for \cref{thm:miller}.

\begin{question}[Kechris--Solecki--Todor\v{c}evi\'c \cite{kst}]
    Let $G$ be a Borel digraph with maximum out-degree $d \in \N$. Is $\chi_B(G) \in \{0,\ldots,2d+1,\aleph_0\}$?
\end{question}

\begin{question}[Kechris--Marks \cite{kechris.marks}]\label{q:km}
    Let $G$ be a Borel graph with maximum out-degree $d \in \N$. Do we have $\chi_\mu(G) \leq 2d+1$ for any Borel probability measure $\mu$ on $G$?
    Do we have $\chi_{BM}(G) \leq 2d+1$ for any compatible Polish topology on $G$?
\end{question}

Note that \cref{thm:kst_main} gives a positive answer for $d = 1$. In the Borel setting,
Palamourdas improved the $3^d$ to $O(d^2)$ and gave a positive answer for $d = 2$ \cite{palamourdas}. 
Recently, Pelayo-G\'omez gave a complete positive answer to the mod-null part of \cref{q:km}.

\begin{theorem}[Pelayo-G\'omez \cite{pelayo-gomez}]\label{thm:pg_main}
    Let $G$ be a Borel graph with maximum out-degree $d \in \N$, and $\mu$ a Borel probability measure on $V(G)$. Then $\chi_\mu(G) \leq 2d+1$. 
\end{theorem}

The purpose of this note is to give a short proof of the following ``randomized Borel'' version of \cref{thm:pg_main}. It implies \cref{thm:pg_main} by Fubini's theorem. 

\begin{theorem}\label{thm:main}
    Let $G$ be a Borel graph with maximum out-degree $d \in \N$. There is a Borel probability space $(\Omega,\mu)$ and a Borel map $c:\Omega \times V(G) \to 2d+1$ such that for each $(x,y) \in G$, for $\mu$-almost every $\omega \in \Omega$, $c_\omega(x) \neq c_\omega(y)$. 
\end{theorem}

The main idea of our proof was adapted from the main idea in \cite{pelayo-gomez}. 
We first need the following. 
It says that $G$ admits a Borel fractional $2^{d+1}$-coloring.

\begin{lemma}
    Let $G$ be a Borel graph with maximum out-degree $d \in \N$. There is a Borel set $I \subseteq 2^\N \times V(G)$
    such that
    \begin{enumerate}
        \item For each $y \in 2^\N$, $I_y$ is $G$-independent.
        \item For each $x \in V(G)$, $I^x$ has coin-flip measure exactly $2^{-d-1}$. 
    \end{enumerate}
\end{lemma}

\begin{proof}
    We prove this with ``exactly $2^{-d-1}$'' replaced by ``at least $2^{-d-1}$.'' We can then get ``exactly'' by passing to a subset. 
    By \cref{cor:kst}, let $f : V(G) \to \N$ be a Borel proper coloring of $G$.
    let \[I_y = \{x \in V(G) \mid \text{$y(f(x)) = 1$ and for all out-neighbors $x'$ of $x$, $y(f(x'))=0$} \}.\] 
    This is $G$-independent by construction, and for each $x \in V(G)$, $I^x$ has measure $2^{-k-1}$, where $k \leq d$ is the number of distinct colors held by $x$'s out-neighbors. 
\end{proof}

\begin{proof}[Proof of \cref{thm:main}]
    We will take $\Omega = ((d+1) \times 2^\N)^\N$ and $\mu$ the product of the uniform measure on $(d+1)$ and the coin-flip measure on $2^\N$. 
    We will inductively define random functions $c_n : V(G) \to 2d+1$ for $n \in \N$, where $c_0$ will be the constant $0$ function, and $c_{n+1}$ will depend uniformly on $c_n$ and the $n$-th coordinate of our random $\omega \in \Omega$. 
    We will then set $c = c_\omega = \lim_n c_n$, or $0$ if the limit does not exist. Call a vertex \textbf{sad} for a function $V(G) \to 2d+1$ if it has the same color as one of its out-neighbors, and \textbf{happy} otherwise.
    We will maintain that for each $n \in \N$ and $x \in V(G)$, 
    \[ \mathbb{P}[\text{$x$ is sad for $c_n$}] \,\leq\, \alpha^n,\]
    where $\alpha = 1 - \frac{2^{-d-1}}{d+1} < 1$.

    Given $c_n$ and $\omega(n) = (i,y) \in (d+1) \times 2^\N$, we define $c_{n+1}$ as follows.
    Let $I$ be the set from the lemma. 
    If $x$ is sad for $c_n$ and $x \in I_y$, let $c_{n+1}(x)$ be the $i$-th color not held by any of $x$'s out-neighbors. Note this is well defined since there are $2d+1$ colors and at most $d$ out-neighbors. 
    In this case let us say $x$ was \textbf{recolored at stage $n$}.
    Otherwise let $c_{n+1}(x) = c_n(x)$. 

    Now let us bound the probability of sadness for $c_{n+1}$. 
    By the independence of $I_y$, if $x$ is recolored at stage $n$, it is happy for $c_{n+1}$. Therefore, there are only two ways for $x$ to be sad for $c_{n+1}$:
    \begin{itemize}
        \item $x$ is sad for $c_n$ and not recolored, i.e., $x \not\in I_y$, or
        \item one of $x$'s out-neighbors is recolored at stage $n$ and chooses $c_n(x)$ as its new color. 
    \end{itemize}
    The former happens with probability $\mathbb{P}[\text{$x$ is sad for $c_n$}] \cdot (1-2^{-d-1})$. 
    For each out-neighbor $x'$ of $x$, the latter happens for this out-neighbor with probability at most
    $\mathbb{P}[\text{$x'$ is sad for $c_n$}] \cdot \frac{2^{-d-1}}{d+1}$. 
    Taking a union bound, we conclude
    \[ \mathbb{P}[\text{$x$ is sad for $c_{n+1}$}] \,\leq\, \alpha^n\left(1-2^{-d-1}\right) + \alpha^n \frac{d 2^{-d-1}}{d+1} \,=\, \alpha^n \left(1 - \frac{2^{-d-1}}{d+1}\right) \,=\, \alpha^{n+1}.\]

    It remains to show that each $x \in V(G)$ is a.s. happy for the limiting coloring $c$. 
    By the Borel--Cantelli lemma, $x$ is a.s. sad for only finitely many $c_n$. 
    Since $x$ is recolored only when it is sad, we have $c(x) = \lim_n c_n(x)$ a.s. The same is true for each of $x$'s out-neighbors. Thus, a.s., there is some $n \in \N$ for which $c= c_n$ on $x$ and all of its out-neighbors and $x$ is happy for $c_n$. Then $x$ is happy for $c$, as desired.
\end{proof}

\noindent\textbf{AI Disclosure:} An LLM was used in the production of this note only to check for typos. 

\medskip

\noindent\textbf{Funding:} This material is based upon work partially supported by the Alfred P. Sloan Foundation and the National Science Foundation under grants DMS-2528522 and DMS-2402064. Any opinions, findings, and conclusions or recommendations expressed in this
material are those of the authors and do not necessarily reflect the views of the funding agencies.

\vspace{-5pt}

\bibliographystyle{amsalpha}
\bibliography{references}

\end{document}